\documentclass[11pt,reqno]{amsart}
\usepackage{amsfonts,amscd}
\usepackage{amssymb}
\usepackage{url}
\usepackage[english]{babel}

\theoremstyle{plain}
\newtheorem{theorem}                 {Theorem}      [section]
\newtheorem{proposition}  [theorem]  {Proposition}

\newtheorem{lemma}        [theorem]  {Lemma}
\newtheorem{conjecture}   [theorem]  {Conjecture}

\theoremstyle{definition}
\newtheorem{example}      [theorem]  {Example}
\newtheorem{remark}       [theorem]  {Remark}
\newtheorem{definition}   [theorem]  {Definition}

\numberwithin{equation}{section}

\def \R{{\mathbb R}}
\def \rn{{\mathbb R}}
\def \s{{\mathbb S}}

\def \C{{\mathbb C}}
\def \cn{{\mathbb C}}

\def \H{{\mathbb H}}

\def \B{\mathcal B}

\def \nab#1#2{\hbox{$\nabla$\kern -.3em\lower 1.0 ex
\hbox{$#1$}\kern -.1 em {$#2$}}}

\def\Re{\mathfrak R\mathfrak e}

\def \g{\mathfrak{g}}

\def \k{\mathfrak{k}}

\def \p{\mathfrak{p}}

\def \un{\mathfrak{u}}

\def \glr#1{\mathfrak{gl}_{#1}(\rn)}
\def \GLC#1{\text{\bf GL}_{#1}(\cn)}
\def \glc#1{\mathfrak{gl}_{#1}(\cn)}

\def \SO#1{\text{\bf SO}(#1)}

\def \SOO#1#2{\text{\bf SO}(#1,#2)}

\def \U#1{\text{\bf U}(#1)}
\def \u#1{\mathfrak{u}(#1)}

\def \SU#1{\text{\bf SU}(#1)}

\def \Sp#1{\text{\bf Sp}(#1)}

\DeclareMathOperator{\Div}{div}

\DeclareMathOperator{\trace}{trace}

\numberwithin{equation}{section}

\begin{document}

\title[Biharmonic functions]
{Biharmonic functions on the \\  special unitary group $\SU 2$}


\author{Sigmundur Gudmundsson}

\address{Mathematics, Faculty of Science\\ University of Lund\\
Box 118, Lund 221\\
Sweden}
\email{Sigmundur.Gudmundsson@math.lu.se}

\allowdisplaybreaks

\begin{abstract}
We construct new proper biharmonic functions defined on open and dense subsets of the special unitary group $\SU 2$.  Then we employ a duality principle to obtain new proper biharmonic functions from the non-compact $3$-dimensional hyperbolic space $\H^3$.
\end{abstract}



\maketitle

\section{Introduction}

The literature on biharmonic functions is vast, but usually the domains are either surfaces or open subsets of flat Euclidean space. The first proper biharmonic functions, from open subsets of the classical compact simple Lie groups $\SU n$, $\SO n$ and $\Sp n$, were recently constructed in \cite{Gud-Mon-Rat-1}.  They are all quotients of linear combinations of the matrix coefficients for the {\it standard} irreducible representation $\pi_1$ of the corresponding group. The authors make use of the well-known fact that these  matrix coefficients are all eigenfunctions of the corresponding Laplace-Beltrami operator.  For this see \cite{Gud-Sve-Ville-1} or Proposition 5.28 of \cite{Kna}.

In this paper we continue the study for the special unitary group $\SU 2$ to its higher finite-dimensional irreducible representations $\pi_n$. In our Theorems \ref{theorem-R2-SU2}, \ref{theorem-R3-SU2} and \ref{theorem-R4-SU2} we construct new proper biharmonic functions. They are all rational functions in the matrix coefficients of the irreducible representations $\pi_2$, $\pi_3$ and $\pi_4$ of $\SU 2$, respectively. When studying these cases an interesting pattern comes to light. This makes us believe that the same method will work for the general finite dimensional irreducible representation $\pi_n$ of $\SU 2$.  These thoughts are formulated in Conjecture \ref{conjecture-Rn-SU2}.  As an introduction to the representations $\pi_n$ of $\SU 2$ we recommend Fegan's book \cite{Feg}.

The special unitary group $\SU 2$ is, up to a constant multiple of the metric, isometric to the $3$-dimensional unit sphere $\s^3$ in the Euclidean $\rn^4$.  This means that our solutions can be seen as functions from open and dense subsets of $\s ^3$.  In Theorem \ref{theorem-duality} we describe a general duality principle and employ this to yield new proper biharmonic functions on the non-compact $3$-dimensional hyperbolic space $\H ^3$.

We conclude the paper with a short appendix providing a formula that hopefully will be useful to implement our calculations by means of some suitable software.

\section{Proper $r$-harmonic functions}\label{section-r-harmonic}

Let $(M,g)$ be a smooth manifold equipped with a semi-Riemannian metric $g$.  We complexify the tangent bundle $TM$ of $M$ to $T^{\cn}M$ and extend the metric $g$ to a complex-bilinear form on $T^{\cn}M$.  Then the gradient $\nabla f$ of a complex-valued function $f:(M,g)\to\cn$ is a section of $T^{\cn}M$.  In this situation, the well-known {\it linear} Laplace-Beltrami operator (alt. tension field) $\tau$ on $(M,g)$ acts on $f$ as follows
$$
\tau(f)=\Div (\nabla f)=\frac{1}{\sqrt{|g|}} \frac{\partial}{\partial x_j}
\left(g^{ij}\, \sqrt{|g|}\, \frac{\partial f}{\partial x_i}\right).
$$
For complex-valued functions $f_1,f_2,h_1,h_2:(M,g)\to\cn$ we have the following well-known relations
\begin{eqnarray}\label{tau-kappa-relations}\nonumber
\tau(f_1h_1)&=&\tau(f_1)\,h_1+2\,\kappa(f_1,h_1)+f_1\tau(h_1),\\
\kappa(f_1h_1,f_2h_2)&=&h_1h_2\,\kappa(f_1,f_2)+f_2h_1\,\kappa(f_1,h_2)\\
& &\qquad +f_1h_2\,\kappa(h_1,f_2)+f_1f_2\,\kappa(h_1,h_2),\nonumber
\end{eqnarray}
where the symmetric {\it conformality} operator $\kappa$ is given by
$$
\kappa(f,h)=g(\nabla f,\nabla h).
$$
For a positive integer $r$, the iterated Laplace-Beltrami operator $\tau^r$ is defined by
$$
\tau^{0} (f)=f,\quad \tau^r (f)=\tau(\tau^{(r-1)}(f)).
$$
\begin{definition}\label{definition-proper-r-harmonic} For a positive integer $r$, we say that a complex-valued function $f:(M,g)\to\cn$ is
\begin{enumerate}
\item[(a)] {\it $r$-harmonic} if $\tau^r (f)=0$, and
\item[(b)] {\it proper $r$-harmonic} if $\tau^r (f)=0$ and $\tau^{(r-1)} (f)$ does not vanish identically.
\end{enumerate}
\end{definition}

It should be noted that the {\it harmonic} functions are exactly the $1$-harmonic
and the {\it biharmonic} functions are the $2$-harmonic ones.  In some texts, the $r$-harmonic functions are also called {\it polyharmonic} of order $r$.

We recall that a map $\pi:(\hat M,\hat g)\to(M,g)$ between two semi-Riemannian manifolds is a \emph{harmonic morphism} if it pulls back germs of harmonic functions to germs of harmonic functions. The standard reference on this topic is the book \cite{BW-book} of Baird and Wood.   We also recommend the updated online bibliography \cite{Gud-bib}.

It was recently shown in \cite{Gud-Mon-Rat-1} that there is an interesting connections between the theory of $r$-harmonic functions and the notion of harmonic morphisms.  In the sequel, we shall often employ the two following results.

\begin{proposition}\cite{Gud-Mon-Rat-1}\label{proposition:lift-tension}
Let $\pi:(\hat M,\hat g)\to (M,g)$ be a submersive harmonic morphism
from a semi-Riemannian manifold $(\hat M,\hat g)$ to a Riemannian manifold
$(M,g)$. Further let $f:(M,g)\to\C$ be a smooth function and
$\hat f:(\hat M,\hat g)\to\C$ be the composition $\hat f=f\circ\pi$.
If $\lambda:\hat M\to\rn^+$ is the dilation of $\pi$ then the tension
field satisfies
$$
\tau(f)\circ\pi=\lambda^{-2}\tau(\hat f)\ \ \text{and}\ \
\tau^r(f)\circ\pi=\lambda^{-2}\tau(\lambda^{-2}\tau^{(r-1)}(\hat f)),
$$
for all positive integers $r\ge 2$.
\end{proposition}

\begin{proposition}\cite{Gud-Mon-Rat-1}\label{proposition:lift-proper-r-harmonic}
Let $\pi:(\hat M,\hat g)\to (M,g)$ be a submersive harmonic morphism,
from a semi-Riemannian manifold $(\hat M,\hat g)$ to a Riemannian manifold
$(M,g)$, with constant dilation.  Further let $f:(M,g)\to\C$ be a smooth function and
$\hat f:(\hat M,\hat g)\to\C$ be the composition $\hat f=f\circ\pi$.
Then the following statements are equivalent
\begin{enumerate}
\item[(i)] $\hat f:(\hat M,\hat g)\to\C$ is proper $r$-harmonic,
\item[(ii)] $f:(M,g)\to\C$ is proper $r$-harmonic.
\end{enumerate}
\end{proposition}

\section{The Riemannian Lie group $\GLC n$}

Let $G$ be a Lie group with Lie algebra $\g$ of left-invariant vector fields on $G$.  Then a Euclidean scalar product $g$ on $\g$
induces a left-invariant Riemannian metric on the group $G$ and turns it into a homogeneous Riemannian manifold. If $Z$ is a left-invariant vector field on $G$ and $f:U\to\cn$ is a complex-valued functions defined locally on $G$ then the first and second order derivatives satisfy
\begin{equation}\label{eq:derivativeZ}
Z(f)(p)=\frac {d}{ds}[f(p\cdot\exp(sZ))]\big|_{s=0},
\end{equation}
\begin{equation}\label{eq:derivativeZZ}
Z^2(f)(p)=\frac {d^2}{ds^2}[f(p\cdot\exp(sZ))]\big|_{s=0}.
\end{equation}

Further, assume that $G$ is a subgroup of the complex general linear group $\GLC n$ equipped with its standard Riemannian metric.  This is induced by the Euclidean scalar product on the Lie algebra $\glc n$ given by
$$
g(Z,W)=\Re\trace ZW^*.
$$
Employing the Koszul formula for the Levi-Civita connection $\nabla$ on $\GLC n$, we see that
\begin{eqnarray*}
g(\nab ZZ,W)&=&g([W,Z],Z)\\
&=&\Re\trace (WZ-ZW)Z^t\\
&=&\Re\trace W(ZZ^t-Z^tZ)^t\\
&=&g([Z,Z^t],W).
\end{eqnarray*}
Let $[Z,Z^t]_\g$ be the orthogonal projection of the bracket $[Z,Z^t]$ onto the subalgebra $\g$ of $\glc n$.  Then the above calculations shows that
$$\nab ZZ=[Z,Z^t]_\g.$$
This implies that the tension field $\tau(f)$ and the conformality operator $\kappa(f,h)$ are given by
\begin{equation}\label{tau-kappa-alie-groups}
\tau(f)=\sum_{Z\in\B}Z^2(f)-[Z,Z^t]_\g(f)
\ \ \text{and}\ \
\kappa(f,h)=\sum_{Z\in\B}Z(f)Z(h),
\end{equation}
where $\B$ is any orthonormal basis for the Lie algebra $\g$.

\begin{remark}
For $1\le i,j\le n$ we shall by
$E_{ij}$ denote the element of $\glr n$ satisfying
$$(E_{ij})_{kl}=\delta_{ik}\delta_{jl}$$ and by $D_t$ the diagonal
matrices $$D_t=E_{tt}.$$ For $1\le r<s\le n$ let $X_{rs}$ and
$Y_{rs}$ be the matrices satisfying
$$X_{rs}=\frac 1{\sqrt 2}(E_{rs}+E_{sr}),\ \ Y_{rs}=\frac
1{\sqrt 2}(E_{rs}-E_{sr}).$$
\end{remark}

\section{The standard irreducible representation $\pi_1$ of $\SU n$}

In this section we describe known proper biharmonic functions on the special unitary group $\SU n$ constructed in \cite{Gud-Mon-Rat-1}. They are quotients of first order homogeneous polynomials in the matrix coefficients of the standard irredible representation $\pi_1$ of $\SU 2$.

The unitary group $\U n$ is the compact subgroup of $\GLC n$ given by
$$
\U n=\{z\in\GLC{n}|\ z\cdot z^*=I_n\}
$$
with its standard matrix representation
$$
\pi^1=
\begin{bmatrix}
z_{11} & z_{12} & \cdots & z_{1n}\\
z_{21} & z_{22} & \cdots & z_{2n}\\
\vdots & \vdots & \ddots & \vdots \\
z_{n1} & z_{n1} & \cdots & z_{nn}
\end{bmatrix}.
$$
The circle group $\s^1=\{e^{i\theta}\in\cn | \ \theta\in\rn\}$ acts on the unitary group $\U n$ by multiplication
$$
(e^{i\theta},z)\mapsto e^{i\theta} z
$$
and the orbit space of this action is the special unitary group
$$
\SU n=\{ z\in\U n|\ \det z = 1\}.
$$
The natural projection $\pi:\U n\to\SU n$ is a harmonic morphism with constant dilation $\lambda\equiv 1$.

The Lie algebra $\u n$ of the unitary group $\U n$ satisfies
$$
\u{n}=\{Z\in\cn^{n\times n}|\ Z+Z^*=0\}
$$
and for this we have the canonical orthonormal basis
$$
\{Y_{rs}, iX_{rs}|\ 1\le r<s\le n\}\cup\{iD_t|\ t=1,\dots ,n\}.
$$

Now, by means of a direct computation based on \eqref{eq:derivativeZ}, \eqref{eq:derivativeZZ} and \eqref{tau-kappa-alie-groups}, we have the following basic result, see \cite{Gud-Sak-1}.

\begin{lemma}\label{lemma:complex}
For $1\le j,\alpha\le n$, let $z_{j\alpha}:\U n\to\cn$ be the complex-valued matrix coefficients of the standard representation of $\U n$ given by
$$
z_{j\alpha}:z\mapsto e_j\cdot z\cdot e_\alpha^t,
$$
where $\{e_1,\dots ,e_n\}$ is the canonical basis for $\cn^n$. Then the following relations hold
\begin{equation}\label{lemma5-1}
\tau(z_{j\alpha})= -n\cdot z_{j\alpha}\ \ \text{and}\ \ \kappa(z_{j\alpha},z_{k\beta})= -z_{k\alpha}z_{j\beta}.
\end{equation}
\end{lemma}

The next result describes the first known proper biharmonic functions from the unitary group $\SU n$, see \cite{Gud-Mon-Rat-1}.

\begin{proposition}\label{proposition-R1-SUn}
Let $p,q\in\cn^n$ be linearly independent and $P,Q:\U n\to\cn$ be the complex-valued functions on the unitary group given by
$$
P(z)=\sum_{j=1}^np_jz_{j\alpha}\ \ \text{and}\ \ Q(z)=\sum_{k=1}^nq_kz_{k\beta}.
$$
Further, let the rational function $f(z)=P(z)/Q(z)$ be defined on the open and dense subset
$W_Q=\{z\in\U n|\ Q(z)\neq 0\}$ of the unitary group. Then the following is true.
\begin{itemize}
\item[(a)] The function $f$ is harmonic if and only if $\alpha=\beta$.
\item[(b)] The function $f$ is proper biharmonic if and only if $\alpha\neq\beta$.
\end{itemize}
The corresponding statements hold for the function induced on $\SU n$.
\end{proposition}

\section{The irreducible representation $\pi_2$ of $\SU 2$}

In this section we extend the construction of Proposition \ref{proposition-R1-SUn} for $\SU 2$ to its $3$-dimensional irreducible representation $\pi_2$ given by the following matrix
$$
\pi^2=
\begin{bmatrix}
z_{11}^2 & z_{11}z_{12} &z_{12}^2\\
2z_{11}z_{21} & z_{11}z_{22}+z_{12}z_{21} &2z_{12}z_{22}\\
z_{21}^2 & z_{21}z_{22} &z_{22}^2
\end{bmatrix}.
$$

\begin{theorem}\label{theorem-R2-SU2}
Let $p,q\in\cn^3$ and $P,Q:\U 2\to\cn$ be the complex-valued functions on the unitary group given by
$$
P(z)=\sum_{j=1}^3p_j\pi^2_{j\alpha}\ \ \text{and}\ \ Q_\beta(z)=\sum_{k=1}^3q_k\pi^2_{k\beta}.
$$
Let the rational functions $f(z)=P(z)/Q(z)$ be defined on the open and dense subset $W_Q=\{z\in\U 2|\ Q(z)\neq 0\}$ of the unitary group.  If $\alpha\neq\beta$ then the function $f$ is proper biharmonic if $p_2q_3-p_3q_2\neq 0$ and
$$p_1q_3^2=q_2(2p_2q_3-p_3q_2),\ \ q_1q_3= q_2^2.$$
The corresponding statement holds for the function induced on $\SU 2$.
\end{theorem}

\begin{proof}
It is easily seen that for a general quotient $f=P/Q$ we have
\begin{equation}\label{tau-quotient}
Q^3\tau(f)=Q^2\tau(P)-2Q\kappa(P,Q)+2P\kappa(Q,Q)-PQ\tau(Q).
\end{equation}
The matrix coefficients for the irreducible representation $\pi_2$ of $\SU 2$ are all eigenfunctions of the tension field $\tau$ with the same eigenvalue.  This is useful since then the above formula \eqref{tau-quotient} simplifies to
\begin{equation*}\label{tau-quotient-simple}
Q^3\tau(f)=-2Q\kappa(P,Q)+2P\kappa(Q,Q).
\end{equation*}

Let us first consider the case when $(\alpha,\beta)=(3,1)$. Then an elementary but lengthy calculation gives the following formula for the tension fields.

\begin{eqnarray*}
\tau(f)&=&8\,\,\frac{(z_{11}z_{22}-z_{12}z_{21})}{(q_1z_{11}^2+2q_2z_{11}z_{21}+q_3z_{21}^2)^2}[(p_1q_2-p_2q_1)z_{11}z_{12}+\\
& &\hskip -.5cm (p_2q_2-p_3q_1)z_{11}z_{22}+(p_1q_3-p_2q_2)z_{12}z_{21}+(p_2q_3-p_3q_2)z_{21}z_{22}]
\end{eqnarray*}
and
$$
\tau^2(f)(z)=-16\,\frac{(z_{11}z_{22}-z_{12}z_{21})^2
(N_1z_{11}^2-N_2z_{11}z_{21}-N_3z_{21}^2)}
{(q_1z_{11}^2+2q_2z_{11}z_{21}+q_3z_{21}^2)^3},
$$
where
$$
N_1=(p_1q_1q_3-4p_1q_2^2+6p_2q_1q_2-3p_3q_1^2),
$$
$$
N_2=(6p_1q_2q_3-8p_2q_1q_3-4p_2q_2^2+6p_3q_1q_2),
$$
$$
N_3=(3p_1q_3^2-6p_2q_2q_3-p_3q_1q_3+4p_3q_2^2).
$$
The bitension field $\tau^2(f)$ vanishes if and only if $N_1=N_2=N_3=0$.  This system of equations is equivalent to
$$p_1q_3^2=q_2(2p_2q_3-p_3q_2)\ \ \text{and}\ \ q_1q_3= q_2^2.$$
By substituting these two relations into the above formula for the tension field $\tau(f)$ we finally obtain
$$
\tau(f)= 8\,\frac{(z_{11}z_{22}-z_{12}z_{21})(q_2z_{12}+q_3z_{22})(p_2q_3-p_3q_2)}{(q_2z_{11}+q_3z_{21})^3}.
$$
This shows that if $(\alpha,\beta)=(3,1)$ then the function $f$ is proper biharmonic if and only if
$$
p_1q_3^2=q_2(2p_2q_3-p_3q_2),\ \ q_1q_3= q_2^2\ \ \text{and}\ \ p_2q_3-p_3q_2\neq 0.
$$
It is easily checked that the above mentioned particular substitutions give the same stated result in all the cases when $\alpha\neq\beta$.
\end{proof}

\begin{example}\label{example-proper-biharmonic-S^3}
Let $p,q\in\cn^3$ be such that
$$
p_1q_3^2=q_2(2p_2q_3-p_3q_2),\ \ q_1q_3= q_2^2\ \ \text{and}\ \ p_2q_3-p_3q_2\neq 0.
$$
Then Theorem \ref{theorem-R2-SU2} tells us that the local function
$$f(z,w)=\frac
{p_1\pi^2_{11}(z,w)+p_2\pi^2_{21}(z,w)+p_3\pi^2_{31}(z,w)}
{q_1\pi^2_{13}(z,w)+q_2\pi^2_{23}(z,w)+q_3\pi^2_{33}(z,w)}
$$
is proper biharmonic on $\SU 2$. It is a well-known fact that $\SU 2$ is, up to a constant conformal factor, isometric to the Lie group $\s^3$ of unit quaternions in the standard Euclidean $\cn^2\cong\rn^4$ via
$$
(z,w)\in\s^3\mapsto
\begin{bmatrix}z & w\\ -\bar w & \bar z\end{bmatrix}
=\begin{bmatrix}x_1+ix_2 & x_3+ix_4\\ -x_3+ix_4 & x_1-ix_2\end{bmatrix}\in\SU 2.
$$
If we identify $x=(x_1,x_2,x_3,x_4)\in\rn$ with $(z,w)=(x_1+ix_2,x_3+ix_4)\in\cn^2$
then we get a proper biharmonic function
$$
f(x)=\frac
{p_1(x_1+ix_2)^2-2\,p_2(x_1+ix_2)(x_3-ix_4)+p_3(x_3-ix_4)^2}
{q_1(x_3+ix_4)^2+2\,q_2(x_1-ix_2)(x_3+ix_4)+q_3(x_1-ix_2)^2}
$$
defined locally on $\s^3$. To confirm this result we have the following alternative argument.

The radial projection $\pi:\R^4\setminus\{0\}\to \s^3$ with $\pi:x\mapsto x/|x|$ is a well-known harmonic morphism and its dilation satisfies $\lambda^{-2}(x)=|x|^2$. If $\hat f$ is the composition $\pi\circ f$ then we can easily calculate the tension fields $\tau(f)$ and $\tau^2(f)$, using Proposition \ref{proposition:lift-tension}, as follows
\begin{eqnarray*}
\tau(f)&=&|x|^2\Delta\hat f\\
&=& -16|x|^2(p_2q_3-p_3q_2)\frac{(q_2(x_1+ix_2)-q_3(x_3-ix_4))}{(q_3(x_1-ix_2)+q_2(x_3+ix_4))^3}\\
&\neq& 0
\end{eqnarray*}
and
$$\tau^2(f)=|x|^2\Delta(|x|^2\Delta(\hat f))=0.$$
Here $\Delta$ is the tension field on $\R^4$ i.e. the classical Laplace operator given by
$$
\Delta = \frac{\partial^2}{\partial x_1^2}+\frac{\partial^2}{\partial x_2^2}+\frac{\partial^2}{\partial x_3^2}+\frac{\partial^2}{\partial x_4^2}.
$$
We will see later on that this induces proper biharmonic functions on the $3$-dimensional hyperbolic space $\H^3$.
\end{example}

\section{The irreducible representation $\pi_3$ of $\SU 2$}

In this section we extend our construction of Theorem \ref{theorem-R2-SU2} to the $4$-dimensional irreducible representation $\pi_3$ of $\SU 2$ given by the matrix
$$
\pi^3=
\begin{bmatrix}
z_{11}^3 & z_{11}^2z_{12} & z_{11}z_{12}^2 & z_{12}^3\\
3z_{11}^2z_{21} & z_{11}(z_{11}z_{22}+2z_{12}z_{21}) & z_{12}(2z_{11}z_{22}+z_{12}z_{21}) & 3z_{12}^2z_{22}\\
3z_{11}z_{21}^2 & z_{21}(2z_{11}z_{22}+z_{12}z_{21}) & z_{22}(z_{11}z_{22}+2z_{12}z_{21}) & 3z_{12}z_{22}^2\\
z_{21}^3 &z_{21}^2z_{22} & z_{21}z_{22}^2 & z_{22}^3
\end{bmatrix}.
$$

\begin{theorem}\label{theorem-R3-SU2}
Let $p,q\in\cn^{4}$ and $P,Q:\U 2\to\cn$ be the complex-valued functions on the unitary group given by
$$
P(z)=\sum_{j=1}^{4}p_j\pi^3_{j\alpha}\ \ \text{and}\ \ Q(z)=\sum_{k=1}^{4}q_k\pi^3_{k\beta}.
$$
Let the rational function $f(z)=P(z)/Q(z)$ be defined on the open and dense subset $W_Q=\{z\in\U 2|\ Q(z)\neq 0\}$ of the unitary group.  If $\alpha\neq\beta$ then the function $f$ is proper biharmonic if $p_3q_4-p_4q_3\neq 0$ and
$$
p_1q_4^3=q_3^2(3p_3q_4-2p_4q_3),\ \ q_1q_4^2=q_3^3,
$$
$$
p_2q_4^2= q_3(2p_3q_4-p_4q_3),\ \ q_2q_4= q_3^2.
$$
The corresponding statement holds for the function induced on $\SU 2$.
\end{theorem}

\begin{proof}
The result can be proven with exactly the same method as we used for Theorem \ref{theorem-R2-SU2}.  But the elementary calculations are more involved in this case.
\end{proof}

\section{The irreducible representation $\pi_4$ of $\SU 2$}

In this section we extend the construction of Theorem \ref{theorem-R2-SU2} for $\SU 2$ to its $5$-dimensional irreducible representation $\pi_4$.  The corresponding matrix $\pi^4$ is given by

{\small
$$
\hskip -1.8cm
\begin{bmatrix}
z_{11}^4 & z_{11}^3z_{12}& z_{11}^2z_{12}^2 & z_{11}z_{12}^3 & z_{12}^4\\
4z_{11}^3z_{21} & z_{11}^2(z_{11}z_{22}+3z_{12}z_{21}) & 2z_{11}z_{12}(z_{11}z_{22}+z_{12}z_{21}) & z_{12}^2( 3z_{11}z_{22}+z_{12}z_{21}) & 4z_{12}^3z_{22}\\
6z_{11}^2z_{21}^2 & 3z_{11}z_{21}(z_{11}z_{22}+z_{12}z_{21}) & z_{11}^2z_{22}^2+4z_{11}z_{12}z_{21}z_{22}+z_{12}^2z_{21}^2 & 3z_{12}z_{22}(z_{11}z_{22}+z_{12}z_{21}) & 6z_{12}^2z_{22}^2\\
4z_{11}z_{21}^3 & z_{21}^2(3z_{11}z_{22}+z_{12}z_{21}) & 2z_{21}z_{22}(z_{11}z_{22}+z_{12}z_{21}) & z_{22}^2(z_{11}z_{22}+3z_{12}z_{21}) & 4z_{12}z_{22}^3\\
z_{21}^4 &z_{21}^3z_{22} & z_{21}^2z_{22}^2 & z_{21}z_{22}^3 & z_{22}^4
\end{bmatrix}
$$
}

\begin{theorem}\label{theorem-R4-SU2}
Let $p,q\in\cn^5$ and $P,Q:\U 2\to\cn$ be the complex-valued functions on the unitary group given by
$$
P(z)=\sum_{j=1}^5p_j\pi^4_{j\alpha}\ \ \text{and}\ \ Q(z)=\sum_{k=1}^5q_k\pi^4_{k\beta}.
$$
Let the rational function $f(z)=P(z)/Q(z)$ be defined on the open and dense subset
$W_Q=\{z\in\U 2|\ Q(z)\neq 0\}
$
of the unitary group. If $\alpha\neq\beta$ then the function is proper biharmonic if $p_4q_5\neq p_5q_4$ and
$$
p_1q_5^4=q_4^3(4p_4q_5-3p_5q_4),\ \ q_1q_5^3=q_4^4,
$$
$$
p_2q_5^3=q_4^2(3p_4q_5-2p_5q_4),\ \ q_2q_5^2=q_4^3,
$$
$$
p_3q_5^2=  q_4(2p_4q_5- p_5q_4),\ \ q_3q_5  =q_4^2.
$$
The corresponding statement holds for the function induced on $\SU 2$.
\end{theorem}

\begin{proof}
The result can be proven with exactly the same method as we used for Theorem \ref{theorem-R2-SU2}.  But the elementary calculations are much more involved in this case.
\end{proof}

\section{The irreducible representation $\pi_n$ of $\SU 2$}

In this section we discuss the general  $(n+1)$-dimensional irreducible representation $\pi_n$ of $\SU 2$.  After investigating $\pi_2,\pi_3$ and $\pi_4$ a clear pattern has appeared.  For this reason we formulate the following conjecture, for which it seems a non-trivial exercise to find a general proof.

\begin{conjecture}\label{conjecture-Rn-SU2}
For $n>1$, let $p,q\in\cn^{n+1}$ and $P,Q:\U 2\to\cn$ be the complex-valued functions on the unitary group given by
$$
P(z)=\sum_{j=1}^{n+1}p_j\pi^n_{j\alpha}\ \ \text{and}\ \ Q(z)=\sum_{k=1}^{n+1}q_k\pi^n_{k\beta}.
$$
Let the rational function $f(z)=P(z)/Q(z)$ be defined on the open and dense subset $W_Q=\{z\in\U 2|\ Q(z)\neq 0\}$ of the unitary group. If $\alpha\neq\beta$ then the function $f$ is proper biharmonic if $p_nq_{n+1}-p_{n+1}q_n\neq 0$ and
$$
p_1q_{n+1}^n=q_n^{n-1}(n\cdot p_nq_{n+1}-(n-1)\cdot p_{n+1}q_n),\ \ q_1q_{n+1}^{n-1}=q_n^n,
$$
$$
p_2q_{n+1}^{n-1}=q_n^{n-2}((n-1)\cdot p_nq_{n+1}-(n-2)\cdot p_{n+1}q_n),\ \ q_2q_{n+1}^{n-2}=q_n^{n-1},
$$
$$
\vdots
$$
$$
p_{n-1}q_{n+1}^2=q_n(2\cdot p_nq_{n+1}- 1\cdot p_{n+1}q_n),\ \ p_{n-1}q_{n+1}=q_n^2.
$$
The corresponding statement holds for the function induced on $\SU 2$.
\end{conjecture}

\section{The Duality}\label{section-duality}

The approach and the methods of this section were introduced in \cite{Gud-Sve-1}. Let $G$ be a non-compact semisimple Lie group with the Cartan decomposition $\g=\k+\p$ of the Lie algebra of $G$ where $\k$ is the Lie algebra of a maximal compact subgroup $K$. Let $G^\cn$ denote the complexification of $G$ and $U$ be the compact subgroup of $G^\cn$ with Lie algebra $\un=\k+i\p$. Let $G^\cn$ and its subgroups be equipped with a left-invariant semi-Riemannian metric which is a multiple of the Killing form by a negative constant.  Then the subgroup $U$ of $G^\cn$ is Riemannian and $G$ is semi-Riemannian.

Let $f:W\to\cn$ be a real analytic function from an open subset $W$ of $G$. Then $f$ extends uniquely to a holomorphic function $f^\cn:W^\cn\to\cn$ from some open subset $W^\cn$ of $G^\cn$. By restricting this to $U\cap W^\cn$ we obtain a real analytic function $f^*:W^*\to\cn$ on some open subset $W^*$ of $U$. The function $f^*$ is called the {\it dual function} of $f$.

\begin{theorem}\cite{Gud-Mon-Rat-1}\label{theorem-duality}
For the above situation we have the following duality.  A complex-valued function $f:W\to\cn$ is proper r-harmonic if and only if its dual $f^*:W^*\to\cn$ is proper r-harmonic.
\end{theorem}

\begin{proof}
A proof is this statement can be found in \cite{Gud-Mon-Rat-1}.
\end{proof}

\begin{remark}\label{remark-duality}
We point out that a function $f:W\to\cn$ is $K$-invariant if and only if its dual $f^*:W^*\to\cn$ is $K$-invariant. In particular, the duality principle of Theorem~\ref{theorem-duality} is valid for the corresponding functions on the quotient spaces.
\end{remark}

We can now apply the duality principle to the local solutions $f$, on the compact Riemannian symmetric space $\s^3=\SO 4/\SO 3$, given in Example \ref{example-proper-biharmonic-S^3}.  This yields local proper biharmonic functions on its non-compact dual i.e. the $3$-dimensional hyperbolic space $\H^3=\SOO 13/\SO 3$.

\begin{example}\label{example-proper-biharmonic-H^3}
Let $\R^4_1$ be the standard $4$-dimensional Minkowski space equipped with its Lorentzian metric
$$
(x,y)_L=-x_0y_0+x_1y_1+x_2y_2+x_3y_3.
$$
Bounded by the light cone, the open set
$$
U=\{x\in\R^4_1|\ (x,x)_L<0\ \text{and}\ 0<x_0\}
$$
contains the 3-dimensional hyperbolic space
$$
\H^3=\{(x_0,x_1,x_2,x_3)\in\R^4_1|\ (x,x)_L=-1\ \text{and}\ 0<x_0\}.
$$
Let $\pi^*:U\to \H^3$ be the radial projection given by
$$
\pi^*:x\mapsto \frac{x}{\sqrt{-(x,x)_L}}.
$$
This is a harmonic morphism and its dilation satisfies $\lambda^{-2}(x)=-|x|^2_L$, see \cite{Gud-7}. Let $p,q\in\cn^3$ and $f$ be the local proper biharmonic function on $\s^3$ defined in Example \ref{example-proper-biharmonic-S^3}.  Then its dual function $f^*:W\to\C$ is defined locally on $\H^3$ with
$$
f^*(x_0,x_1,x_2,x_3)=f(-ix_0,x_1,x_2,x_3).
$$
Then
$$
f(x)=\frac
{p_1(-ix_0+ix_1)^2+2\,p_2(-ix_0+ix_1)(-x_2+ix_3)+p_3(x_2-ix_3)^2}
{q_1(x_2+ix_3)^2+2\,q_2(-ix_0-ix_1)(x_2+ix_3)+q_3(-ix_0-ix_1)^2}.
$$
Let $\hat f^*$ be the composition $f^*\circ\pi^*$ from the Minkowski space $\rn^4_1$.  Then
\begin{eqnarray*}
\tau(f^*)&=&-|x|_L^2\Box\hat f^*\\
&=& 16|x|_L^2(p_2q_3-p_3q_2)\frac{(q_2(x_0-x_1)-iq_3(x_2-ix_3))}{(q_3(x_0+x_1)+iq_2(x_2+ix_3))^3}   \\
&\neq& 0
\end{eqnarray*}
and
$$\tau^2(f^*)=-|x|_L^2\Box(-|x|_L^2\Delta(\hat f))=0.$$
Here $\Box$ is the tension field on $\R^4_1$ i.e. the wave operator of d'Alembert given by
$$
\Box = -\frac{\partial^2}{\partial x_0^2}+\frac{\partial^2}{\partial x_1^2}+\frac{\partial^2}{\partial x_2^2}+\frac{\partial^2}{\partial x_3^2}.
$$
From this we see that $f^*$ is proper biharmonic. It should be noted that $q\in\cn^3$ can easily be chosen such that $f^*:\H^3\to\cn$ is globally defined.
\end{example}

\section{Acknowledgements}
The author is grateful to Stefano Montaldo and Andrea Ratto for useful discussions related to this work.

\appendix

\section{}\label{general-F(z)}

The complex general linear group $\GLC n$ is an open subset of $\cn^{n\times n}\cong\cn^{n^2}$ and inherits its standard complex structure.  As before, let $z_{ij}$ with $1\leq i,\,j \leq n$ denote the matrix coefficients of a generic element $z\in\GLC n$.
Let $G$ be a Lie subgroup of $\GLC n$, $F:U\to\cn$ be a holomorphic function defined on an open subset of $\GLC n$ and $f:U\cap G\to\cn$ be the restriction of $F$ to $G$.

Following Lemma 3.2 of \cite{Gud-5}, the equations \eqref{eq:derivativeZ}, \eqref{eq:derivativeZZ} and \eqref{tau-kappa-alie-groups} can be utilized to compute the tension field of $f$, see also \cite{Gud-Mon-Rat-1}.
\begin{equation*}\label{tension-general-F}
\tau(f)= \sum_{1\leq i,j,k,\ell \leq n} \frac {\partial ^2 f}{\partial z_{ij}\partial z_{k\ell}}\,\kappa(z_{ij},z_{k\ell}) \,+\,
\sum_{1\leq i,j\leq n} \frac {\partial  f}{\partial z_{ij} }\,\tau (z_{ij}).
\end{equation*}
For the unitary group $\U n$ we can apply Lemma \ref{lemma:complex} to make the above formula more explicit
\begin{equation*}\label{tension-general-F-su-U(n)}
\tau(f)=-\sum_{1\leq i,j,k,\ell \leq n}\frac{\partial ^2 f}{\partial z_{ij}\partial z_{k\ell}}\,z_{kj}z_{i\ell}
\ -n\sum_{1\leq i,j\leq n}\frac{\partial  f}{\partial z_{ij}}\,z_{ij}.
\end{equation*}
This formula is our main tool for carrying out computer calculations for checking the results of this paper.

\end{document}